\documentclass[11pt]{article}

\usepackage{amsmath,amsthm,amssymb}
\usepackage{mathtools}

\usepackage[english]{babel}
\usepackage[T1]{fontenc}
\usepackage{color}
\usepackage{hyperref}

\usepackage[title]{appendix}

\newtheorem{theorem}{Theorem}
\newtheorem{corollary}[theorem]{Corollary}
\newtheorem{lemma}[theorem]{Lemma}

\newtheorem*{remark}{Remark}

\newtheorem{proposition}[theorem]{Proposition}

\newcommand{\N}{\mathbb{N}}

\newcommand{\R}{\mathbb{R}}
\newcommand{\C}{\mathbb{C}}
\newcommand{\EE}{\mathbb{E}}
\newcommand{\PP}{\mathbb{P}}

\newcommand{\di}{\,\mbox{d}}

\DeclareMathOperator{\lh}{span}

\allowdisplaybreaks

\def\1{{\mathchoice {1\mskip-4mu\mathrm l}   
{1\mskip-4mu\mathrm l}
{1\mskip-4.5mu\mathrm l} {1\mskip-5mu\mathrm l}}}

\def\hh{\widehat{h}}
\def\dd{{\rm d}}

\begin{document}
\title{Persistence exponents via perturbation theory: MA(1)-processes}
\author{\renewcommand{\thefootnote}{\arabic{footnote}}Frank Aurzada\footnotemark[1]\and \renewcommand{\thefootnote}{\arabic{footnote}}Dieter Bothe\footnotemark[1]\and \renewcommand{\thefootnote}{\arabic{footnote}}Pierre-\'Etienne Druet\footnotemark[1]\and \renewcommand{\thefootnote}{\arabic{footnote}}Marvin Kettner\footnotemark[1]\and \renewcommand{\thefootnote}{\arabic{footnote}}Christophe Profeta\footnotemark[2]}
\footnotetext[1]{Technical University of Darmstadt, Schlossgartenstra\ss e 7, 64287 Darmstadt, Germany. E-mail: aurzada@mathematik.tu-darmstadt.de, bothe@mma.tu-darmstadt.de, druet@mma.tu-darmstadt.de, marvinkettner@web.de}
\footnotetext[2]{Universit\'e Paris-Saclay, CNRS, Univ Evry, Laboratoire de Math\'ematiques et Mod\'elisation d'Evry, 91037, Evry-Courcouronnes, France. E-mail: christophe.profeta@univ-evry.fr}

\maketitle

\begin{abstract}
For the moving average process $X_n=\rho \xi_{n-1}+\xi_n$, $n\in\N$, where $\rho\in\R$ and $(\xi_i)_{i\ge -1}$ is an i.i.d.\ sequence of normally distributed random variables,\ we study the persistence probabilities $\PP(X_0\ge 0,\dots, X_N\ge 0)$, for $N\to\infty$. We exploit that the exponential decay rate $\lambda_\rho$ of that quantity, called the persistence exponent, is given by the leading eigenvalue of a concrete integral operator. This makes it possible to study the problem with purely functional analytic methods. In particular, using methods from perturbation theory, we show that the persistence exponent $\lambda_\rho$ can be expressed as a power series in $\rho$. Finally, we consider the persistence problem for the Slepian process, transform it into the moving average setup, and show that our perturbation results are applicable.
\end{abstract}




\section{Introduction}\label{sec: introduction}
\subsection{Persistence probabilities for moving average processes}
Let $(\xi_i)_{i\ge -1}$ be a sequence of i.i.d.\ random variables and $\rho\in \R$. A moving average process of order one (MA(1)-process) is given by
$$
X_n := \rho \xi_{n-1} + \xi_n, \quad \text{for } n\in\N. 
$$
Throughout the paper, we consider standard normally distributed random variables $(\xi_i)_i$, i.e.\ $\xi_0$ has the density $\phi(x):= ( 2\pi)^{-1/2} \exp( -x^2/2)$, $x\in\R$.

We are interested in the persistence probabilities
\begin{equation}
    \label{eqn:persistence}
\PP(X_0\ge 0,\dots,X_N \ge 0), \qquad\text{for } N\to\infty,
\end{equation}
and in particular in the exponential rate of decay of this quantity, which is called persistence exponent. Non-exit probabilities are a classic and fundamental topic in probability with numerous applications in finance, insurance, queueing, and other subjects. The question is also studied intensively in theoretical physics; there, the rationale is that the persistence exponent is a measure of how fast the underlying physical system returns to equilibrium. We refer to the survey \cite{bray} and the monograph \cite{metzler} for an overview on the relevance of the question to physical systems and to \cite{aurzadasimon} for a survey of the mathematical literature.

The persistence problem for moving average processes was studied before in \cite{MajBra,krishnapur,AurMukZei}. For MA-processes, the persistence question can be rewritten as a non-exit problem for a two-dimensional Markov chain (see e.g.\ \cite[Section~2.1]{AurMukZei}). It is well-known that non-exit probabilities of Markov chains have close relations to eigenvalues of operators (see \cite{AurBau,AurMukZei,HinKolWac}, also see \cite{ChaVil,Collet,Meleard,Tweedie2,Tweedie1} for the quasi-stationary approach). The purpose of this paper is establish and to use the connection between the persistence exponent $\lambda_\rho$, i.e.\ the exponential decay rate of (\ref{eqn:persistence}), and the leading eigenvalue of a suitable operator. Then, powerful tools from functional analysis can be applied. In particular, methods of perturbation theory in the spirit of \cite{Kato} can be used to obtain a series representation of the eigenvalue in the parameter $\rho$. This ansatz was previously used in \cite{AurKet} for autoregressive processes.

\subsection{The eigenvalue problem}
In the recent work \cite[Section 2.1]{AurMukZei} it is shown that for $\rho\neq -1$ the persistence probability (\ref{eqn:persistence}) of the MA(1)-process  decays exponentially fast and that the exponential decay rate, i.e.\ the persistence exponent, is the leading eigenvalue of the following explicit integral operator:
\begin{equation} \label{eqn:mstartt}
    S_\rho\colon \mathcal{B}(\R)\to \mathcal{B}(\R), \quad S_\rho f(x):= \int_{-\rho x}^\infty f(y)\phi(y)\di y,
\end{equation}    
where $\mathcal{B}(\R)$ is the space of bounded measurable real-valued functions on $\R$. Bearing this connection in mind, the purpose of this paper is to study the largest solution $\lambda=\lambda_\rho$ of the eigenvalue equation
\begin{equation}   \label{eqn:eigenvalueequation}
\lambda f(x) = \int_{-\rho x}^\infty f(y)\phi(y)\di y,\qquad x\in\R, f \in \mathcal{B}(\R).
\end{equation}

The approach taken in this paper is to show that a modification of the integral operator $S_\rho$ can be represented as a power series in $\rho$ if we consider the operator on a suitable space of functions. The definition of this function space (see Section~\ref{sec: main results}) is motivated by the following observation:

Assume that $f$ is analytic, that is $f(x)=\sum_{n=0}^\infty \frac{f^{(n)}(0)}{n!} x^n$, for $x\in\R$. Further, if we assume that $\lim_{x\to\infty}(f\phi)^{(n-1)}(x) =0$ for all $n\ge 1$, we can write $(f\phi)^{(n-1)}(0)= (-1)\int_0^\infty (f\phi)^{(n)}(y)\di y$. Then,

\begin{align}
S_\rho f(x) &=\int_{-\rho x}^\infty f(y)\phi(y)\di y  \nonumber\\
&= \int_0^\infty f(y)\phi(y)\di y + (-1) \int_0^{-\rho x} f(y) \phi(y)\di y \nonumber \\
&=  \int_0^\infty f(y) \phi(y) \di y + \sum_{n=1}^\infty \rho^n (-1)^{n+1} x^n \frac{(f\phi)^{(n-1)}(0)}{n!} \nonumber \\
&= \int_0^\infty f(y) \phi(y) \di y + \sum_{n=1}^\infty \rho^n (-1)^{n} x^n \frac{1}{n!}\int_0^\infty (f\phi)^{(n)}(y)\di y \nonumber\\
&= \sum_{n=0}^\infty \rho^n (-1)^{n} x^n \frac{1}{n!} \int_0^\infty (f\phi)^{(n)}(y)\di y.  \label{eqn: repr. T_rhof(x)}
\end{align}
Hence, under the above conditions on $f$, the expression $S_\rho f(x)$ can be written as a power series in $\rho$.

With this observation in mind, we will consider a specific space of analytic functions such that we obtain a well-defined holomorphic operator (cf.\ Theorem~\ref{thm: operator holomorphic MA}). From this, by using perturbation techniques in the spirit of \cite{Kato}, we can conclude that the leading eigenvalue and the corresponding eigenfunction are holomorphic in $\rho$, too. This is the contents of our main result, Theorem~\ref{thm: connection to persistence problem MA}. In other words, the persistence exponent and the eigenfunction admit a power series representation in $\rho$, respectively.
Additionally, we have iterative formulas for the coefficients of the power series representation of the persistence exponent and the eigenfunction, respectively; and we compute the first coefficients, cf.\ Theorem~\ref{thm: iterative formula MA}.

\subsection{Application to the Slepian process}
As a further application, we look at the persistence problem for the so-called Slepian process. Let $(B_t)_{t\ge 0}$ be a standard Brownian motion and define the Slepian process by
\begin{align}
S_t := B_{t+1}-B_t, \qquad t\geq 0. \label{eqn:slepstart}
\end{align}
In his seminal paper \cite{Slepian}, D.\ Slepian computed the distribution of the supremum of $S$ on $[0,1]$ and found:
\begin{equation}\label{eq: Slepian01}
\PP\left(\sup_{u\in[0,1]}S_u\leq a\right)=  \Phi(a)^2 - \phi(a)\left(a\Phi(a)+\phi(a)\right),
\end{equation}
where again $\phi$ is the standard normal density and $\Phi$ is the corresponding cumulative distribution function.
The general formula for the distribution of $\sup_{u\in[0,t]} S_u$ for any $t\geq 0$ was then obtained by Shepp in \cite{Shepp}. Shepp leaves it as an open question to study of the asymptotics  
\begin{align}\label{eq: asymptotics Slepian process}
\lim_{N\to\infty} \frac{1}{N} \log \PP\left(\sup_{u\in[0,N]}S_u \leq a \right),
\end{align}
because his formulas, which involve iterated integrals, are not well-suited for such computations. The existence of this limit was then obtained by Li \& Shao in \cite{LiSh}, and numerical computations have been proposed recently by Noonan \& Zhigljavsky \cite{NoZh}. 
The contribution of this paper is the observation that one can rewrite the persistence problem for the Slepian process as a persistence problem for an MA(1)-process. Further, we can show that our perturbation results are applicable.

\medskip
The outline of this paper is as follows. In Section~\ref{sec: main results}, our main results are stated. In particular, we show that the leading eigenvalue $\lambda_\rho$ of (\ref{eqn:eigenvalueequation}) can be expanded into a power series in $\rho$ and we discuss the coefficients of this power series. In Section~\ref{sec: Slepian}, we consider the persistence problem for the Slepian process, transform it into the MA setup, and show that our main results are applicable.
Section~\ref{sec: proofs} is devoted to the proofs of the main theorems. Finally, in Section~\ref{sec:radius} we deal with the radius of convergence of the series for $\lambda_\rho$.

%

\section{Main results}\label{sec: main results}

Let $\gamma$ be the standard Gaussian measure on $\R$ and let $h_n$ denote the $n$-th Hermite polynomial given by 
\[h_n(x):= (-1)^n e^{\frac{x^2}{2}}\frac{\di^n}{\di x^n} e^{-\frac{x^2}{2}}.\]
For the facts on Hermite polynomials that we use in this paper, we refer the reader to \cite[Chapter 6]{Andrews}.
Further we set $\hh_n(x):= (n!)^{-1/2} h_n(x)$. Here, the normalization is chosen such that $\|\hh_n\|_{L^2(\R,\gamma)}=1$.

Fix $0<q<1$ and let $(a_n)_{n\in\N}\subseteq \C$ be a sequence such that $\sum_{n=0}^\infty |a_n|^2 q^{-n}<\infty$.
By \cite[Theorem 2]{Ali}, it holds that $\sum_{n=0}^\infty |a_n \hh_n (x)|$ converges uniformly on compact subsets of $\R$. Hence, we can define an analytic function $f\colon\R\to\C$ via $f(x):= \sum_{n=0}^\infty a_n \hh_n(x)$. In particular, $\Re(f)$ and $\Im(f)$ are analytic, where $\Re(f)$ and $\Im(f)$ are the real and the imaginary part of $f$, respectively. 
Let
\[\mathcal{H}_q := \Big\{f\colon\R\to\C,\ f(x)=\sum_{n=0}^\infty a_n \hh_n(x) \text{ with } \sum_{n=0}^\infty |a_n|^2 q^{-n} <\infty \Big\}\]
and set 
\[\langle f, g\rangle_{\mathcal{H}_q} := \sum_{n=0}^\infty a_n\overline{b_n} q^{-n} \text{  for }  f=\sum_{n=0}^\infty a_n \hh_n,\ g=\sum_{n=0}^\infty b_n \hh_n.\]
We note that $(\mathcal{H}_q,\langle \cdot,\cdot\rangle_{\mathcal{H}_q})$ is a Hilbert space of functions \cite[Proposition 1]{Ali}. In fact, we will see that it is a reproducing kernel Hilbert space and we will exploit this structure for the proofs. Note that we consider a complex Hilbert space instead of a real one, since a complex space is necessary for applying the powerful methods of perturbation theory in the spirit of \cite{Kato}.

We set
\[
T_\rho\colon \mathcal{H}_q\to \mathcal{H}_q, \quad T_\rho f(x):=\int_{-\rho x}^\infty f(y)\phi(y)\di y.
\]
Note that this is the version of the operator $S_\rho$ on the space $\mathcal{H}_q$, cf.\ (\ref{eqn:mstartt}), in the sense that $S_\rho$ acts on bounded real valued functions, while $T_\rho$ acts on complex valued functions in $\mathcal{H}_q$. Here, we set, for a complex valued function $f$: $\int_{-\rho x}^\infty f(y)\phi(y)\di y = \int_{-\rho x}^\infty \Re(f)(y)\phi(y)\di y + i\int_{-\rho x}^\infty \Im(f)(y)\phi(y)\di y$. 

Our first main result states that $T_\rho$ is a holomorphic operator.

\begin{theorem}\label{thm: operator holomorphic MA}
Fix $0<q<1$. Let $-\sqrt{\frac{1-q}{1+q^{-1}}}< \rho < \sqrt{\frac{1-q}{1+q^{-1}}}$ and define for $n\in\N$ the integral operator
\[T^{(n)}\colon \mathcal{H}_q \to \mathcal{H}_q, \quad T^{(n)}f(x) := (-1)^n x^n \frac{1}{n!}\int_0^\infty (f\phi)^{(n)}(y)\di y.\]
The operator $T_\rho$ is well-defined, bounded, compact and admits the representation
\begin{equation}
    \label{eqn:operatorisholomorphic}
T_\rho = \sum_{n=0}^\infty \rho^n T^{(n)}.
\end{equation}
\end{theorem}

\begin{remark}
To optimize the radius of convergence of the power series for $T_\rho$, the best choice of $0<q<1$ for the Hilbert space $\mathcal{H}_q$ is $q^*:= \sqrt{2}-1$. Then, $T_\rho$ can be represented as a power series for $-(\sqrt{2}-1) < \rho< \sqrt{2}-1$.
\end{remark}
Let
\[r(T_\rho):= \sup\{|\lambda| \colon \lambda\in \Sigma(T_\rho)\}\]
be the spectral radius of $T_\rho$, where $\Sigma(T_\rho)$ denotes the spectrum of $T_\rho$.

Our second main result deals with the leading eigenvalue $\lambda_\rho$ of (\ref{eqn:eigenvalueequation}) on $\mathcal{H}_q$, i.e.\ the persistence exponent for moving average processes, and states that $\lambda_\rho$ can be expanded into a power series.

\begin{theorem}\label{thm: connection to persistence problem MA}
For $-\sqrt{\frac{1-q}{1+q^{-1}}} < \rho < \sqrt{\frac{1-q}{1+q^{-1}}}$ we have
\[\PP(X_0\ge 0,\dots, X_N\ge 0) = \lambda_\rho^{N+o(N)},\]
where $\lambda_\rho := r(T_\rho)\in (0,1)$ is the largest eigenvalue of $T_\rho$. The corresponding eigenfunction $f_\rho$ is non-negative, i.e.\ $f_\rho(x)\ge 0$ for all $x\in\R$.

There are numbers $K_n\in\R$ such that the quantity $\lambda_\rho$ admits the representation
\[
\lambda_\rho = \sum_{n=0}^\infty \rho^n K_n,
\]
for all $|\rho|< r_0$, where $r_0>0.332$.
\end{theorem}

As an application, we are going to see in Section~\ref{sec: Slepian} that one can transform the persistence problem for the Slepian process into the setup of persistence of MA(1)-processes and that this case can be covered by Theorem~\ref{thm: connection to persistence problem MA}, cf.\ Proposition~\ref{prop: Slepian}.

%

As the third and last main result, we determine the coefficients $K_n$, $n\in\N$, of the power series of the persistence exponent $\lambda_\rho$. 

By Theorem~\ref{thm: operator holomorphic MA} the operator $T_\rho$ and by Theorem~\ref{thm: connection to persistence problem MA} the eigenvalue $\lambda_\rho$ can be expressed as a power series in $\rho$, respectively. Additionally, the corresponding eigenfunction $f_\rho$ can be expressed as a power series in $\rho$ (see \cite[Theorem 4]{AurKet}). Let us write

\begin{align*}
T_\rho = \sum_{k=0}^\infty \rho^k T^{(k)},
\qquad
\lambda_\rho = \sum_{k=0}^\infty \rho^k K_k,
\qquad
f_\rho = \sum_{m=0}^\infty \rho^m g_m.
\end{align*}

Note that for $\rho=0$ the MA(1)-process is a sequence of i.i.d.\ random variables, so that \[K_0=\lambda_0=\PP(\xi_0\ge 0)=\frac{1}{2}.\]

\begin{theorem}\label{thm: iterative formula MA}
For all $m\in\N$ the function $g_m$ is a polynomial of degree at most $m$. A full iterative description of the $g_m$, $m\in\N$, is given by the equations $g_0=\1$, $K_0=\frac{1}{2}$, $g_m(0)=0$ for $m\ge 1$, and
\begin{equation}\label{eqn:iterative2}
g_m = \frac{1}{K_0} \left( \sum_{j=1}^m T^{(j)} g_{m-j} - \sum_{j=1}^{m-1} T^{(0)} g_j \cdot g_{m-j} \right), \quad m\ge 1.
\end{equation}
Further, the $K_n$, $n\ge 1$, can be computed using 
\begin{equation} \label{eqn:newreprK}
K_n = T^{(0)} g_n.
\end{equation}
The first coefficients are given by \allowdisplaybreaks
\begin{align*}
K_0 &=\frac{1}{2},\\
K_1 &= \frac{1}{\pi},\\
K_2 &= - \frac{2}{\pi^2},\\
K_3 &= - \frac{5}{6\pi} + \frac{8}{\pi^3},\\
K_4 &=  \frac{13}{3 \pi^2}  - \frac{40}{\pi^4},\\
K_5 &= \frac{23}{40\pi} - \frac{28}{\pi^3} + \frac{224}{\pi^5},\\
K_6 &= - \frac{1069}{180\pi^2} +  \frac{580}{3\pi^4}  - \frac{1344}{\pi^6},\\
K_7 &=  -\frac{37}{112 \pi} + \frac{842}{15\pi^3}- \frac{4144}{3\pi^5} + \frac{8448}{\pi^7},\\
K_8 &=  \frac{943}{168 \pi^2} - \frac{1535}{3\pi^4} + \frac{10080}{\pi^6} - \frac{54912}{\pi^8} .
\end{align*}
\end{theorem}

It would be very interesting to obtain a closed-form expression for the coefficients $(K_n)$. We suspect that the last term of each $K_n$, respectively, is given by $\tau_n:=(-1)^{n-1} 2^{n-1}\binom{2(n-1)}{n-1} \frac{1}{n} \pi^{-n}$. The second to last term of the $K_n$, respectively, seems to be of the form $-\tau_{n-2} \frac{ 8 (n - 3) + 5}{6}$. Obtaining more terms seems complicated.

%

\section{Slepian process}\label{sec: Slepian}

In this section, we consider the persistence problem for the Slepian process, show how this can be transformed into a persistence question for an MA(1)-process, and prove that the above main results can be applied.

Recall that the Slepian process $(S_t)$ was defined in (\ref{eqn:slepstart}). Let us denote by 
$$
F_N(a):=\PP\left(\sup_{u\in[0,N]}S_u\leq a\right)
$$
the persistence probability for the Slepian process, where $a\in\R$ and $N\in\N$. 

\begin{proposition}\label{prop: Slepian}
The persistence probability $F_N(a)$ may be written as
$$
F_N(a) = \PP(X_0\leq b, \ldots, X_{N-1}\leq b)
$$
where $(X_n)_{n\geq1}$ is a MA(1)-process with standard normally distributed random
variables and with parameter 
$$\widehat{\rho} = \frac{\sqrt{1-4\cos^2(2\pi F_2(\widehat{a}))} -1}{2\cos(2\pi F_2(\widehat{a})) }\simeq 0.3186, $$
where $\widehat{a} = F_1^{-1}(1/2)$ and  $b=\Phi^{-1}(F_1(a))\sqrt{1+\widehat{\rho}^2}$. 
\end{proposition}

When $b=0$, i.e.\ $a=\widehat{a}$, Theorem~\ref{thm: connection to persistence problem MA} shows that the desired exponential decay rate of the persistence probability of the Slepian process can be expressed as a power series since $\widehat{\rho}<0.332$. For arbitrary $a\in\R$ a shifted version of $T_{\widehat{\rho}}$ needs to be considered. Presumably, as in Theorem~\ref{thm: operator holomorphic MA} and Theorem~\ref{thm: connection to persistence problem MA}, a power series representation can be obtained.

\begin{remark}
The value of $F_1$ was originally computed by \cite{Slepian} and is given in (\ref{eq: Slepian01}).  Similarly, following the computation of \cite{NoZh}, the value of $F_2(\widehat{a})$ is given by
\begin{multline*}
F_2(\widehat{a})=  \Phi(  \widehat{a}) -\Phi^3( \widehat{a}) 
- \frac{1}{\sqrt{2 }}\int_0^{\infty} \Phi( \widehat{a}-y)\varphi(\sqrt{2} y) \left(\Phi(\sqrt{2}y) - \frac{1}{2}\right)\,\emph{d}y\\
+ \frac{\varphi^2( \widehat{a})}{2} \left(\Phi( \widehat{a})( \widehat{a}^2+1) +  \widehat{a}\varphi( \widehat{a})\right)
+\int_0^{\infty} \Phi^2(\widehat{a}-y) \varphi(y+ \widehat{a}) \,\emph{d} y.
\end{multline*}
\end{remark}

\begin{proof}[Proof of Proposition~\ref{prop: Slepian}]
We start by writing the decomposition:
\[
F_N(a)= \PP\left(M_0\leq a, \ldots, M_{N-1}\leq a\right),
\]

where the random variables $(M_i)_{i\geq 0}$ are defined by 
\[
M_i := \sup_{u\in[0,1]} S_{u+i}.
\]
 Using the stationarity  and the independence of the increments of Brownian motion, it is clear that the random variables $(M_i)$ are identically distributed with common distribution $F_1$, and are such that for any $i$ and $j$ with $|i-j|\geq 2$, $M_i$ and $M_j$ are independent.
Define next the centered Gaussian random variables
\[
Z_i :=  \Phi^{-1}(F_1(M_i)), \quad i\ge 0,
\]
and observe that the sequence $(Z_i)$ is stationary, with covariance matrix $V$ given by 

\[
V=\begin{pmatrix}
1 &s &0&0 & \ldots \\
s & 1 & s  & 0&  \ldots   \\
0 & s & 1 & s& \ldots\\
0 & 0 & s& 1 &  \\
\vdots & \vdots & \vdots &  &\ddots
\end{pmatrix},
 \qquad \text{where }s = \EE[Z_0Z_1].
\]

\noindent
Setting $b_0:=\Phi^{-1}(F_1(a))$, we are thus led to compute 
\begin{equation}\label{eq:FnNew}
F_N(a) = \PP\left(Z_0 \leq b_0, Z_1 \leq b_0 , \ldots,  Z_{N-1} \leq b_0\right), 
\end{equation}
where the density of the Gaussian vector $(Z_0, \ldots, Z_{N-1})$ depends only on $s$. To compute the value of $s$, observe that taking $N=2$ and $b_0=0$,  we have
\begin{align*}
F_2\left(\widehat{a}\right) &=  \PP\left(Z_0 \leq 0, Z_1 \leq 0\right)\\
&=  \frac{\sqrt{1-s^2}}{2\pi } \int_{-\infty}^0 \int_{-\infty}^0  \exp\left(-\frac{1}{2} \left(x^2+y^2-2s x y\right)\right) \di x \di y\\
&= \frac{1}{4} + \frac{1}{2\pi} \text{arctan}\left(\frac{s}{\sqrt{1-s^2}}\right).
\end{align*}


Inverting this relation yields the value of $s$ (note that $2\pi F_2(\widehat{a})\in[\frac{\pi}{2},\pi]$ so that we can use the last formula):
$$s = \frac{ \tan\left(2\pi F_2\left(\widehat{a}\right)- \frac{\pi}{2}\right)}{\sqrt{1+ \tan^2\left(2\pi F_2\left(\widehat{a}\right)- \frac{\pi}{2}\right)}}= - \cos\left(2\pi F_2\left(\widehat{a}\right)\right).$$
Finally, let $(\xi_i)_{i\geq -1}$ be an i.i.d.\  sequence of random variables with distribution $\xi_0\sim\mathcal{N}(0,1)$ and set
\[\widehat{\rho} := \frac{1-\sqrt{1-4s^2}}{2s}\qquad \text{ so that }\quad s = \frac{\widehat{\rho}}{1+\widehat{\rho}^2}.\]
Therefore, the sequence $(Z_i)$ has the same distribution as the sequence
\[
\frac{\widehat{\rho}\xi_{i-1} +\xi_i}{\sqrt{1+\widehat{\rho}^2}},\qquad i\geq 0;
\]
and going back to (\ref{eq:FnNew}), we obtain 
\begin{align*}
F_N(a) = \PP\left( \widehat{\rho}\xi_{-1} +  \xi_0 \leq b_0\sqrt{1+\widehat{\rho}^2},\ldots, \widehat{\rho}\xi_{N-2} +  \xi_{N-1} \leq b_0\sqrt{1+\widehat{\rho}^2} \right).
\end{align*}
This is exactly the statement of Proposition~\ref{prop: Slepian} after setting $X_i := \widehat{\rho}\xi_{i-1} + \xi_i $ for $i\in \N$.
\end{proof}

\section{Proofs of the main theorems}\label{sec: proofs}
The following lemma makes it legitimate to use the computation \eqref{eqn: repr. T_rhof(x)} and provides a helpful representation of the inner product and the norm on $\mathcal{H}_q$.

\begin{lemma}\label{lemma: properties of funtions in H_q}
Let $f,g\in\mathcal{H}_q$. It holds that 
\begin{enumerate}
\item[(a)] $\|f^{(k)}\|_{L^2(\R,\gamma)} <\infty$, for all $k\in\N$,
\item[(b)] $\lim_{x\to\infty}(f\phi)^{(n-1)}(x) =0$, for all $n\ge 1$,
\item[(c)] $\langle f,g \rangle_{\mathcal{H}_q} = \sum_{k=0}^\infty \frac{(q^{-1}-1)^k}{k!}\langle f^{(k)},g^{(k)} \rangle_{L^2(\R,\gamma)}$,
\item[(d)] $\|f\|^2_{\mathcal{H}_q} = \sum_{k=0}^\infty \frac{(q^{-1}-1)^k}{k!} \|f^{(k)}\|^2_{L^2(\R,\gamma)}$.
\end{enumerate}
\end{lemma}

\begin{proof}
(a)\ Let $f\in\mathcal{H}_q$, i.e.\ we suppose that $f(x)=\sum_{n=0}^\infty a_n \hh_n(x)$ for $x\in\R$, with $\sum_{n=0}^\infty |a_n|^2 q^{-n}<\infty$. Note that $\hh_n^{(k)}(x)=\sqrt{\frac{n!}{(n-k)!}}\hh_{n-k}(x)$ for $k\le n$ (see~e.g.\ \cite[Section 6.1]{Andrews}). Thus, the derivatives of $f$ are given by 
\[f^{(k)}(x) = \sum_{n=0}^\infty a_n \hh_n^{(k)}(x) = \sum_{n=0}^\infty a_{n+k}\sqrt{\frac{(n+k)!}{n!}}\hh_n(x).\] Recall that $(\hh_n)_n$ is an orthonormal basis for $L^2(\R,\gamma)$. By Parseval's identity we get 
\[\|f^{(k)}\|_{L^2(\R,\gamma)} = \sum_{n=0}^\infty \left(|a_{n+k}|\sqrt{\frac{(n+k)!}{n!}}\right)^2 \le \sum_{n=0}^\infty |a_{n+k}|^2 (n+k)^{k} < \infty.\]

(b)\ Note that $(f\phi)^{(n-1)} = (\Re(f)\phi)^{(n-1)} + i(\Im(f)\phi)^{(n-1)}$. We have
\begin{align*}
\int_0^\infty |(\Re(f) \phi)^{(n-1)}(x)|\di x &= \int_0^\infty \Big| \sum_{k=0}^{n-1} \binom{n-1}{k} \Re(f)^{(k)}(x) \phi^{(n-1-k)}(x) \Big| \di x\\
&=  \int_0^\infty \Big| \sum_{k=0}^{n-1} \binom{n-1}{k} \Re(f)^{(k)}(x)h_{n-1-k}(x) \Big| \di \gamma(x)\\
&\le \sum_{k=0}^{n-1} \binom{n-1}{k}\int_\R |\Re(f)^{(k)}(x) h_{n-1-k}(x)| \di\gamma(x) \\
&\le \sum_{k=0}^{n-1} \binom{n-1}{k} \|\Re(f)^{(k)}\|_{L^2(\R,\gamma)} \|h_{n-1-k}\|_{L^2(\R,\gamma)} \\
&<\infty,
\end{align*}
by using H\"older's inequality in the last but one step and statement (a) in the last step.
Therefore, if $\lim_{x\to\infty}(\Re(f)\phi)^{(n-1)}(x)$ exists, then the limit must be zero.
The limit exists since
\begin{align*}
(\Re(f)\phi)^{(n-1)}(x) &=  \int_0^x (\Re(f)\phi)^{(n)}(y)\di y + (\Re(f)\phi)^{(n-1)}(0)\\
&\overset{x\to\infty}{\longrightarrow} \int_0^\infty (\Re(f)\phi)^{(n)}(y) \di y + (\Re(f)\phi)^{(n-1)}(0).
\end{align*}
We conclude similarly that $\lim_{x\to\infty}(\Im(f)\phi)^{(n-1)}(x)=0$. Hence, the assertion follows.

(c) Let $f(x)=\sum_{n=0}^\infty a_n \hh_n(x),\ g(x)=\sum_{n=0}^\infty b_n \hh_n(x)$ for $x\in\R$. As in the proof of statement (a) we have $f^{(k)}(x)= \sum_{n=k}^\infty a_n \sqrt{\frac{n!}{(n-k)!}}\hh_{n-k}(x)$ and $g^{(k)}(x)= \sum_{n=k}^\infty b_n \sqrt{\frac{n!}{(n-k)!}}\hh_{n-k}(x)$ for $x\in\R$ and $k\in\N$. Using that $(\hh_n)_n$ is an orthonormal basis for $L^2(\R,\gamma)$, we compute
\begin{align*}
\sum_{k=0}^\infty \frac{(q^{-1}-1)^k}{k!}\langle f^{(k)},g^{(k)} \rangle_{L^2(\R,\gamma)} &= \sum_{k=0}^\infty \frac{(q^{-1}-1)^k}{k!} \sum_{n=k}^\infty \frac{n!}{(n-k)!}a_n \overline{b_n}\\
&= \sum_{n=0}^\infty a_n \overline{b_n} \sum_{k=0}^n \binom{n}{k} (q^{-1}-1)^k \\
&= \sum_{n=0}^\infty a_n \overline{b_n} (q^{-1}-1+1)^n\\
&= \langle f,g \rangle_{\mathcal{H}_q}.
\end{align*}

(d) This statement follows directly from (c).
\end{proof}

Combining Lemma~\ref{lemma: properties of funtions in H_q}(b) with the fact that functions of $\mathcal{H}_q$ are analytic, it holds that, as in \eqref{eqn: repr. T_rhof(x)}, 
\[T_\rho f(x)= \sum_{n=0}^\infty \rho^n (-1)^{n} x^n \frac{1}{n!} \int_0^\infty (f\phi)^{(n)}(y)\di y \]
for all $f\in\mathcal{H}_q$ and $x\in\R$, i.e.\ (\ref{eqn:operatorisholomorphic}) holds.

\begin{proof}[Proof of Theorem~\ref{thm: operator holomorphic MA}]
We want to compute an upper bound for the operator norm of $T^{(n)}$ for $n\in\N$, which simultaneously shows that these operators are well-defined. 
Let $m_n(x):= x^n$, for $x\in\R$. Note that
\begin{align*}
\|T^{(n)}f\|_{\mathcal{H}_q} &= \|(-1)^n\frac{1}{n!}\int_0^\infty (f \phi)^{(n)}(y)\di y \cdot m_n \|_{\mathcal{H}_q}\\
&= \left|\frac{1}{n!}\int_0^\infty (f \phi)^{(n)}(y)\di y \right| \|m_n\|_{\mathcal{H}_q}.
\end{align*}

We recall the inverse explicit formula for the Hermite polynomials (see e.g.\ \cite[Section 2]{Patarroyo}):
\[ m_n = n! \sum_{j=0}^{\lfloor\frac{n}{2} \rfloor} \frac{h_{n-2j}}{2^j j! (n-2j)!} = \sum_{j=0}^{\lfloor\frac{n}{2} \rfloor} \frac{n!}{2^j j! \sqrt{(n-2j)!}}\hh_{n-2j}.\]
On one hand, by using the inequality $(2j)! \le 2^{2j}j!^2$ for $j\in\N$, we have

\begin{align*}
\|m_n\|_{\mathcal{H}_{q}}^2 &=  \sum_{j=0}^{\lfloor\frac{n}{2} \rfloor} \left(\frac{n!}{2^j j! \sqrt{(n-2j)!}}\right)^2 q^{-(n-2j)}\\
&= n! \sum_{j=0}^{\lfloor\frac{n}{2} \rfloor} \frac{n!}{ 2^{2j} j!^2 (n-2j)!}  q^{-(n-2j)}\\
&\le n! \sum_{j=0}^{\lfloor\frac{n}{2} \rfloor}  \frac{n!}{(2j)!  (n-2j)!}  q^{-(n-2j)}\\
&= n! \sum_{j=0}^{\lfloor\frac{n}{2} \rfloor} \binom{n}{2j}  q^{-(n-2j)}\\
&\le n! \sum_{j=0}^{n} \binom{n}{j}  q^{-(n-j)}\\
&= n! (1+q^{-1})^n.
\end{align*}
Hence,
\[\|m_n\|_{\mathcal{H}_q} \le  \sqrt{ n!} \cdot (1+q^{-1})^{n/2} .\]

On the other hand, we get, by using H\"older's inequality and Lemma~\ref{lemma: properties of funtions in H_q}(d),
\begin{align*}
& \quad\ \left|\frac{1}{n!}\int_0^\infty (f \phi)^{(n)}(y)\di y \right| \\
&\le \frac{1}{n!} \int_0^\infty \sum_{k=0}^n\binom{n}{k} |f^{(k)}(y)\phi^{(n-k)}(y)| \di y\\
&= \frac{1}{n!} \sum_{k=0}^n \binom{n}{k} \int_0^\infty  |f^{(k)}(y)h_{n-k}(y)\phi(y)| \di y\\
&\le \frac{1}{n!} \sum_{k=0}^n \binom{n}{k} \int_\R  |f^{(k)}(y)h_{n-k}(y)| \di \gamma(y)\\
&\le \frac{1}{n!} \sum_{k=0}^n \binom{n}{k} \|f^{(k)}\|_{L^2(\R,\gamma)}\cdot \|h_{n-k}\|_{L^2(\R,\gamma)}\\
&=  \sum_{k=0}^n \frac{1}{k!(n-k)!}\|f^{(k)}\|_{L^2(\R,\gamma)} \sqrt{(n-k)!}\\
&= \sum_{k=0}^n  \frac{(q^{-1}-1)^{k/2}}{\sqrt{k!}} \|f^{(k)}\|_{L^2(\R,\gamma)} \cdot \frac{1}{\sqrt{k! (n-k)!}(q^{-1}-1)^{k/2}}\\
&\le \left(\sum_{k=0}^n  \frac{(q^{-1}-1)^k}{k!} \|f^{(k)}\|_{L^2(\R,\gamma)}^2 \right)^{1/2} \cdot  \left(\sum_{k=0}^n \frac{1}{k! (n-k)!(q^{-1}-1)^k}\right)^{1/2}\\
&\le \|f\|_{\mathcal{H}_q}  \left(\frac{1}{n!}\sum_{k=0}^n \binom{n}{k}\frac{1}{(q^{-1}-1)^k} \right)^{1/2} \\
&=\|f\|_{\mathcal{H}_q}  \,\frac{1}{\sqrt{n!}} \left(\frac{1}{1-q}\right)^{n/2}.
\end{align*}

Taking these computations together, we obtain
\begin{align} \label{eqn:opnormalplusestimatewould}
\|T^{(n)}\| \le \left(\frac{1}{1-q}\right)^{n/2} \cdot (1+q^{-1})^{n/2} = \left(\frac{1+q^{-1}}{1-q}\right)^{n/2}.
\end{align}

Therefore, we get that for $|\rho| < \sqrt{\frac{1-q}{1+q^{-1}}}$ 
\[\sum_{n=0}^\infty \|\rho^n T^{(n)} \| < \infty.\]
The set of all linear and bounded operators on $\mathcal{H}_q$ is a Banach space and thus, $T_\rho = \sum_{n=0}^\infty \rho^n T^{(n)}$ is a linear and bounded operator on $\mathcal{H}_q$. 

For the compactness of $T_\rho$, note that $T^{(n)}$ is a finite-rank operator for all $n\in\N$, namely of rank $1$, i.e.\ the range of $T^{(n)}$ is one-dimensional. As a finite-rank operator, $T^{(n)}$ is a compact operator. The subset of all compact operators in the Banach space of the linear and bounded operators on $\mathcal{H}_q$ is itself a Banach space (see e.g.\ \cite[III.\ Theorem 4.7]{Kato}). Hence, $T_\rho = \sum_{n=0}^\infty \rho^n T^{(n)}$ is compact.
\end{proof}


\begin{proof}[Proof of Theorem~\ref{thm: connection to persistence problem MA}]
Let $-\sqrt{\frac{1-q}{1+q^{-1}}} < \rho < \sqrt{\frac{1-q}{1+q^{-1}}}$. We begin by relating the eigenvalue problem of $T_\rho$ to the persistence problem of the MA(1)-process. First, note that $S_\rho^N(\1)= T_\rho^N(\1)$ for all $N\in\N$. By \cite[Section~2.1]{AurMukZei} we can rewrite the persistence probability as follows:
\begin{align*}
\PP(X_0 \ge 0,\dots,X_N\ge 0)  = \int_{S} T_\rho^N(\1)(x_2)\di(\gamma\otimes\gamma)(x_1,x_2),
\end{align*}
with $S:=\left\{(x_1,x_2)\in\R^2\colon \rho x_1 + x_2 \ge 0\right\}$.
Let $r(T_\rho)$ be the spectral radius of $T_\rho$. We need to show that
\[\int_{S} T_\rho^N(\1)(x_2)\di(\gamma\otimes\gamma)(x_1,x_2) = r(T_\rho)^{N+o(N)}.\]
A priori we cannot exclude that $r(T_\rho)=0$. 
For the upper bound note that
\[\|f\|_{L^1(\R,\gamma)} \le \|f\|_{L^2(\R,\gamma)} \le \|f\|_{\mathcal{H}_q}\]
for all $f\in\mathcal{H}_q$ due to Lemma~\ref{lemma: properties of funtions in H_q}(d). Using this, we obtain

\begin{align*}
\int_{S} T_\rho^N(\1)(x_2)\di(\gamma\otimes\gamma)(x_1,x_2) &\le \|T_\rho^N(\1)\|_{L^1(\R,\gamma)}\\
&\le  \|T_\rho^N(\1)\|_{\mathcal{H}_q} \\
&\le   \|T_\rho^N\| \cdot \|\1\|_{\mathcal{H}_q}\\
&= r(T_\rho)^{N+o(N)},
\end{align*}
where the last step is due to Gelfand's formula, i.e.\ $r(T_\rho)= \lim_{N\to\infty}\|T_\rho^N\|^{\frac{1}{N}}$.
Now, we turn to the lower bound.
We need to consider two cases. If $r(T_\rho)=0$, then clearly
\[
\int_{S} T_\rho^N(\1)(x_2)\di(\gamma\otimes\gamma)(x_1,x_2) \ge r(T_\rho)^{N}.
\]

If $r(T_\rho)>0$, we show the lower bound by using the Krein-Rutman theorem (see \cite{KreRut}, \cite[Theorem 19.2]{Deimling}). For this purpose, let us define the cone $C:= \{f\in \mathcal{H}_q\colon f(x)\ge 0 \text{ for all }x\in\R\}$. 
From \cite[Proposition 1]{Ali} it follows that $\mathcal{H}_q$ is a reproducing kernel Hilbert space with reproducing kernel
\[K_q(x,y):= \sum_{n=0}^\infty q^n \hh_n(x)\hh_n(y)= \frac{1}{\sqrt{1-q^2}}e^{-\frac{q^2x^2+ q^2y^2 -2q xy}{2(1-q^2)}},\]
where the last equality is due to Mehler's formula (see \cite{Mehler}). It holds that $K_q^y(\cdot):= K_q(\cdot,y)\in C$ for all $y\in\R$. Since $\lh\{K_q^y\colon y\in \R\}$ is dense in $\mathcal{H}_q$ (see e.g.\ \cite{Aronszajn}), the closure of $C+(-C)$ is equal to $\mathcal{H}_q$. Further, we have $T_\rho(C)\subseteq C$. Therefore, the Krein-Rutman theorem can be applied and yields the existence of an eigenfunction $g\in C$ with eigenvalue $r(T_\rho)$. Note that any eigenfunction of $T_\rho$ is bounded since
$$
|T_\rho f(x)| \le \int_{-\rho x}^\infty |f(y)\phi(y)|\di y
\le \|f\|_{L^1(\R,\gamma)}
\le \|f\|_{\mathcal{H}_q},
$$
for all $f\in\mathcal{H}_q$ and $x\in\R$. Hence, $\|g\|_{\infty}<\infty$.
We obtain
\begin{align*}
\int_{S} T_\rho^N(\1)(x_2)\di(\gamma\otimes\gamma)(x_1,x_2) &\ge \int_{S} T_\rho^N\left(\frac{g}{\|g\|_{\infty}}\right)(x_2)\di(\gamma\otimes\gamma)(x_1,x_2)\\
&= r(T_\rho)^N  \int_{S} \frac{g(x_2)}{\|g\|_{\infty}}\di(\gamma\otimes\gamma)(x_1,x_2)\\
&= r(T_\rho)^{N+o(N)}.
\end{align*}
Thus,
\[\PP(X_0\ge 0,\dots,X_N\ge 0 ) = r(T_\rho)^{N+o(N)}.\]
Now, note that \cite[Proposition 2.3]{AurMukZei} implies that $r(T_\rho)>0$. Hence, \linebreak$\lambda_\rho := r(T_\rho)>0$ is the largest eigenvalue of $T_\rho$ by the Krein-Rutman theorem. 
Further, we have $\PP(X_0\ge 0,\dots, X_N\ge 0)\le \PP(\min_{0\le n \le \lfloor \frac{N}{2}\rfloor}X_{2n}\ge 0)$. 
Note that the random variables $\{X_{2n} \colon 0\le
n\le  \lfloor \frac{N}{2}\rfloor\}$ are independent. Hence,
\[\PP(X_0\ge 0,\dots, X_N\ge 0) \le \PP(X_0\ge 0)^{\lfloor \frac{N}{2}\rfloor +1},\]
which implies $\lambda_\rho<1$.

It remains to show that $\lambda_\rho$ admits a representation as a power series. We will see that this follows by \cite[Theorem 4]{AurKet}, which is based on the classical work of \cite{Kato}.
From the eigenvalue equation 
\begin{align}\label{eq: eigenvalue eq. T_0}
\lambda f(x)= T_0 f(x)= \1 \int_0^\infty f(y)\phi(y) \di y \quad\text{for all } x\in\R,
\end{align}
it follows that the largest eigenvalue of $T_0$ is given by $\lambda_0=\frac{1}{2}$. To obtain the analyticity of the eigenvalue at $0$ in $\rho$ by methods of perturbation theory, it is necessary to show that the algebraic multiplicity of $\lambda_0$ is equal to one.

For this purpose, let $P_{\lambda_0}$ be the spectral projection of $\lambda_0$. The algebraic multiplicity is defined by the dimension of $P_{\lambda_0}L^2(\R,\gamma)$. Due to the compactness of $T_0$, we get $P_{\lambda_0}L^2(\R,\gamma) = \ker(\lambda_0-T_0)^{v}$, where $v\in\N$ is the smallest natural number such that $\ker(\lambda_0-T_0)^{v} = \ker(\lambda_0-T_0)^{v+1}$ (see e.g.\ \cite{Conway}).\\
From the eigenvalue equation \eqref{eq: eigenvalue eq. T_0}, we see that $\ker(\lambda_0-T_0)^{1}$ is equal to the constant functions and therefore one-dimensional.\\
If we prove that $\ker(\lambda_0-T_0)^{1} = \ker(\lambda_0-T_0)^{2}$, it follows that the algebraic multiplicity of $\lambda_0$ is one. Let $g\in \ker(\lambda_0-T_0)^{2}$. Then, 
\begin{align*}
0 &= (\lambda_0-T_0)^{2} (g)\\
&= (\lambda_0 -T_0)(\lambda_0 g - T_0 (g))\\
&= \lambda_0^2 g -2 \lambda_0 T_0 (g) + T_0(T_0 (g)).
\end{align*}
Since $T_0 (g)$ is constant, the above equation yields that $g$ is constant, i.e.\ $g\in\ker(\lambda_0-T_0)^1$.
By \cite[Theorem 4]{AurKet}, it follows that $\lambda_\rho$ can be represented as a power series for $|\rho|<r_0$ for some $r_0>0$. The question of the radius of convergence is tricky and will be tackled in Section~\ref{sec:radius}. In particular, the only thing that is still to be proved is the bound for the radius of convergence, which can be found in Corollary~\ref{cor:radius} combined with Lemma~\ref{lem:bothrepresenationsareidentical}.
\end{proof}

\begin{remark}
    We remark that the operator $T_\rho$ is not normal. If it were, further results from perturbation theory would be applicable. In particular, a concrete bound for the radius of convergence would follow from Corollary~6 in \cite{AurKet} together with the bound (\ref{eqn:opnormalplusestimatewould}).
\end{remark}

\begin{proof}[Proof of Theorem~\ref{thm: iterative formula MA}]
The eigenvalue equation $T_\rho (f_\rho) = \lambda_\rho f_\rho$ reads:
$$
\sum_{k=0}^\infty \rho^k T^{(k)} \left( \sum_{m=0}^\infty \rho^m g_m\right) = \sum_{k=0}^\infty \rho^k K_k \cdot \sum_{m=0}^\infty \rho^m g_m.
$$
Sorting this in powers of $\rho$ (which is allowed within the radius of convergence) gives
$$
\sum_{n=0}^\infty \rho^n \sum_{k=0}^n T^{(k)} g_{n-k}  = \sum_{n=0}^\infty \rho^n \sum_{k=0}^n K_k  g_{n-k}.
$$
Since this holds for any $\rho\in(-\rho_0,\rho_0)$ with $\rho_0>0$, we must have
\begin{equation} \label{eqn:relationsKg}
\sum_{k=0}^n T^{(k)} g_{n-k}  = \sum_{k=0}^n K_k  g_{n-k},\qquad \text{for all $n\in\N$.}
\end{equation}
For $n=0$, this is
$$
T^{(0)} g_0 = K_0 g_0.
$$
Since the left-hand side is a constant, by the definition of $T^{(0)}=T_0$, we know that $g_0$ must be constant, so w.l.o.g.\ (multiplication of the eigenfunction) take
$$
g_0=\1.
$$

Fix $n\ge 1$. We now analyse the iterative structure of \eqref{eqn:relationsKg}:
\begin{equation} \label{eqn:polynn}
\sum_{k=1}^n T^{(k)} g_{n-k} + T^{(0)} g_n = K_n g_0 + \sum_{k=1}^{n-1} K_k g_{n-k} + K_0 g_n.
\end{equation}
Observe that, by the definition of the operators $T^{(k)}$, the first term on the left is a polynomial. Further, the second term on the left and the first term on the right are constants. The second term on the right involves only the $g_\ell$, $\ell<n$. Therefore, inductively we know that $g_n$ is a polynomial of degree at most $n$.

Let us denote the coefficients of the polynomials $g_\ell$ by $g_\ell^i$, i.e.\ $g_\ell(x)=:\sum_{i=0}^\ell g_\ell^i x^i$. Then comparing the coefficients of $x^0$ in \eqref{eqn:polynn} gives
$$
0 + T^{(0)} g_n = K_n g_0 + \sum_{k=1}^{n-1} K_k g_{n-k}^0 + K_0 g_n^0.
$$
Further, by the definition of $T^{(0)}$ and the represenation of $g_n$, we have
$$
T^{(0)} g_n = \sum_{i=0}^n g_n^i \int_0^\infty y^i \phi(y) \di y.
$$
Combining the last two equations and using that $\int_0^\infty y^0 \phi(y) \di y = K_0$, we see that the terms involving $g_n^0$ cancel; giving
\begin{equation} \label{eqn:constants0}
\sum_{i=1}^n g_n^i \int_0^\infty y^i \phi(y) \di y = K_n g_0 + \sum_{k=1}^{n-1} K_k g_{n-k}^0.
\end{equation}
None of the other polynomial terms in \eqref{eqn:polynn} uses $g_n^0$ either. Therefore, $g_n^0$ is in fact arbitrary and can be chosen to be zero. This however simplifies  \eqref{eqn:polynn} in the sense that
\begin{equation} \label{eqn:iterative1none}
g_n = \frac{1}{K_0} \left( \sum_{k=1}^n T^{(k)} g_{n-k} - \sum_{k=1}^{n-1} K_k g_{n-k} \right),
\end{equation}
because these are the parts of \eqref{eqn:polynn} involving the coefficients of $x^i$, $i>0$, only. Note that \eqref{eqn:iterative1none} computes $g_n$ with the help of $g_\ell$ and $K_\ell$ for $\ell<n$ only.

Similarly, now \eqref{eqn:constants0} simplifies (using $g_\ell^0=0$ and $g_0=\1$) to \eqref{eqn:newreprK}. Putting this back into \eqref{eqn:iterative1none} shows \eqref{eqn:iterative2}.
\end{proof}

\section{Radius of convergence} \label{sec:radius}

This section is devoted to the study of the radius of convergence of the series in Theorem~\ref{thm: connection to persistence problem MA}. In particular, we would like to finish the proof of that theorem by showing that the radius of convergence is at least $0.332$.

For this purpose, we first give an alternative description of the leading eigenvalue of the eigenvalue equation (\ref{eqn:eigenvalueequation}), which might be interesting in its own right, cf.\ (\ref{eqn:newrepresentationA}). We also show that the two descriptions must have the same radius of convergence. Using the alternative description, we can prove a lower bound for the radius of convergence, cf.\ Corollary~\ref{cor:radius}.

We start with a re-formulation of the eigenvalue equation (\ref{eqn:eigenvalueequation}) that gets rid of the eigenfunction.

\begin{lemma} \label{lem:differentrepresentation} 
Let $\rho\in[0,1]$. The leading eigenvalue $\lambda=\lambda_\rho$ of the eigenvalue equation (\ref{eqn:eigenvalueequation}) is the largest positive root of the equation
\begin{equation} \label{eqn:eigenvalueeqn2}
    \lambda  =  \sum_{k=0}^{\infty} \frac{\kappa_k(\rho)}{\lambda^k},
\end{equation}
    where $\kappa_0\equiv\frac{1}{2}$ and for $k\geq 1$:
\begin{eqnarray}
\kappa_k(\rho) &:=&  \frac{1}{(2\pi)^{\frac{k+1}{2}}}   \int_0^{\infty} \int_{-\rho s_0}^0 \ldots \int_{-\rho s_{k-2}}^0 \int_{-\rho s_{k-1}}^0 \exp\left(-\frac{1}{2}\sum_{i=0}^k s_i^2\right)  \dd s_k\ldots \dd s_0  \label{eqn:defn.f_k}
\\
&=&  \frac{\rho^{\frac{k(k+1)}{2}}}{(2\pi)^{\frac{k+1}{2}}}   \int_0^{\infty} \int_{-s_0}^0 \ldots \int_{-s_{k-2}}^0 \int_{-s_{k-1}}^0 \exp\left(-\frac{1}{2}\sum_{i=1}^{k} (\rho^i s_i)^2\right) e^{-\frac{s_0^2}{2}}  \dd s_k\ldots \dd s_0. \notag 
\end{eqnarray}
\end{lemma}

Before we give the proof of Lemma~\ref{lem:differentrepresentation}, we note a useful technical bound for the coefficients $(\kappa_k)$.

\begin{lemma} \label{lem:boundf_k} For the $(\kappa_k)$ defined in (\ref{eqn:defn.f_k}), we have for all $k\geq 0$:
\begin{align}
|\kappa_k(\rho)|
=&
\frac{\rho^{\frac{k(k+1)}{2}}}{(2\pi)^{\frac{k+1}{2}}}   \int_0^{\infty} \int_0^{s_0} \ldots \int_0^{s_{k-2}} \int_0^{s_{k-1}} \exp\left(-\frac{1}{2}\sum_{i=1}^{k} (\rho^i s_i)^2\right) e^{-\frac{s_0^2}{2}}  \dd s_k\ldots \dd s_0 \label{eqn:moddefn.f_k2}
\\
\leq&
\frac{\rho^{\frac{k(k+1)}{2}}}{2^{k+1} \pi^{\frac{k}{2}}}  \frac{1}{\Gamma\left(\frac{k}{2}+1\right)}. \notag
\end{align}
\end{lemma}

\begin{proof}[Proof of Lemma~\ref{lem:boundf_k}]
Bounding all the exponentials in (\ref{eqn:moddefn.f_k2}) by one, except the last one, we obtain:
$$
|\kappa_k(\rho)|
\leq
\frac{\rho^{\frac{k(k+1)}{2}}}{(2\pi)^{\frac{k+1}{2}}k!}  \int_0^{\infty} s^{k} e^{-\frac{s^2}{2}} \dd s 
=
\frac{\rho^{\frac{k(k+1)}{2}}}{2 \pi^{\frac{k+1}{2}}k!}  \Gamma\left(\frac{k+1}{2}\right)
=
\frac{\rho^{\frac{k(k+1)}{2}}}{2^{k+1} \pi^{\frac{k}{2}}}  \frac{1}{\Gamma\left(\frac{k}{2}+1\right)},
$$
where we used the Legendre duplication formula $\Gamma(2t)=2^{2t-1}\Gamma(t)\Gamma(t+1/2)/\sqrt{\pi}$ for $t=(k+1)/2$ in the last step.
\end{proof}

\begin{remark}
A result similar to Lemma~\ref{lem:boundf_k} (and subsequently a result similar to Lemma~\ref{lem:differentrepresentation}) can be proved in the same way as long as the density $\phi$ has a superexponential decay.
\end{remark}

\begin{remark}
We cannot expect explicit values for the coefficients $\kappa_k(\rho)$. In fact, they are persistence probabilities themselves:
$$
\kappa_k(\rho)=\PP( 0\leq \xi_0 \leq \rho \xi_1 \leq \ldots \leq \rho^k \xi_k),
$$
where the $(\xi_i)$ are independent random variables with density $\phi$.
\end{remark}

\begin{proof}[Proof of Lemma~\ref{lem:differentrepresentation}]
The only purpose of the proof is to transform the eigenvalue equation:
\begin{equation} \label{eqn:langrangebuermannequationansatz1}
\lambda f(x) = \int_{-\rho x}^\infty f(s)\phi(s) \dd s = \int_0^\infty f(s)\phi(s) \dd s+\int_{-\rho x}^0 f(s)\phi(s) \dd s,\qquad x\in\R.
\end{equation}
Further, we set w.l.o.g.\ $f(0):=1$. Inserting this into (\ref{eqn:langrangebuermannequationansatz1}) gives
\begin{equation} \label{eqn:langrangebuermannequationan01}
\lambda  =  \int_0^\infty f(s)\phi(s) \dd s.
\end{equation}
Replacing the corresponding term in (\ref{eqn:langrangebuermannequationansatz1}) gives
\begin{equation} \label{eqn:langrangebuermannequationan2}
f(x) =1 +\lambda^{-1} \int_{-\rho x}^0 f(s)\phi(s) \dd s .
\end{equation}
Inserting (\ref{eqn:langrangebuermannequationan2}) into (\ref{eqn:langrangebuermannequationan01}) we obtain
\begin{eqnarray*}
\lambda &=&\int_0^\infty  \left[ 1+\lambda^{-1} \int_{-\rho s_0}^0 f(s_1)\phi(s_1) \dd s_1 \right] \phi(s_0) \dd s_0 
\\
&=& \int_0^\infty \phi(s_0) \dd s_0 +\lambda^{-1}\int_0^\infty \int_{-\rho s_0}^0 f(s_1)\phi(s_1) \dd s_1  \phi(s_0) \dd s_0
\\
&=& \frac{1}{2} +\lambda^{-1}\int_0^\infty \int_{-\rho s_0}^0 f(s_1)\phi(s_1) \dd s_1  \phi(s_0) \dd s_0.
\end{eqnarray*}
Iterating this procedure gives
$$
\lambda =\sum_{k=0}^{N-1} \lambda^{-k} \kappa_k(\rho) + \lambda^{-N} \int_0^\infty \int_{-\rho s_0}^0 \ldots \int_{-\rho s_{N-1}}^0 f(s_N) \phi(s_N) \dd s_N \ldots \phi(s_1) \dd s_1  \phi(s_0) \dd s_0 .
$$
We now let $N\to\infty$ and show that the remainder term vanishes. For this purpose, we use that the eigenfunction $f$ is bounded and we employ the bound from Lemma~\ref{lem:boundf_k}:
\begin{align*}
& \left| \lambda^{-N} \int_0^\infty \int_{-\rho s_0}^0 \ldots \int_{-\rho s_{N-1}}^0 f(s_N) \phi(s_N) \dd s_N \ldots \phi(s_1) \dd s_1  \phi(s_0) \dd s_0 \right| 
\\
\leq &
||f||_\infty \cdot \frac{\mbox{const}^N}{\Gamma(N/2+1)} \to 0. \qquad\qedhere{}
\end{align*}
\end{proof}

We are going to re-write the eigenvalue equation in the following way. Define
$$
\Psi_\rho(\xi):=\sum_{k=1}^\infty 2^{k+1} \kappa_k(\rho) (1+\xi)^{-k}.
$$
Note that finding solutions of (\ref{eqn:eigenvalueeqn2}) for $\lambda\geq\frac{1}{2}$ is equivalent to finding solutions of the following equation for $\xi\geq 0$:
\begin{equation} \label{eqn:eigenvalueeqn3}
    \xi =  \Psi_\rho(\xi),\qquad \lambda=\frac{1}{2}( 1+\xi).
\end{equation}
That means, we are looking for the largest root $\xi\geq 0$ of equation (\ref{eqn:eigenvalueeqn3}). Next, we define
$$
\Gamma_\rho(\xi):=\frac{\xi}{\Psi_\rho(\xi)},\qquad \xi\geq 0.
$$
Let us state some important properties of the function $\Gamma_\rho$ that will allow us to invert the function.
\begin{lemma}
    We have $\Gamma_\rho(0)=0$, $\lim_{\xi\to\infty}\Gamma_\rho(\xi)=\infty$. Further, for $\rho\in[0,1]$, the function $\xi \mapsto \Gamma_\rho(\xi)$ is strictly increasing.
\end{lemma}

\begin{proof} The first two properties follow immediately from Lemma~\ref{lem:boundf_k}. To see the last property, note that the derivative of $\Gamma_\rho$ has the same sign as:
\begin{align}
\Psi_\rho(\xi) - \xi \Psi_\rho^\prime(\xi) &= \sum_{k=1}^{\infty} 2^{k+1} \kappa_k(\rho)(1+\xi)^{-k}  +\sum_{k=1}^{\infty} 2^{k+1} \kappa_k(\rho)k \xi(1+\xi)^{-k-1} \notag\\
&= 4 \kappa_1(\rho) \frac{1+2\xi}{(1+\xi)^2}+ \sum_{k=2}^{\infty}  2^{k+1} \kappa_k(\rho) \frac{1+(k+1)\xi}{(1+\xi)^{k+1}} \notag
\\
&\geq  4 \kappa_1(\rho) \frac{1+2\xi}{(1+\xi)^2}   - \sum_{k=2}^{\infty}  \frac{\rho^{\frac{k(k+1)}{2}}}{\pi^{\frac{k}{2}}}  \frac{1}{\Gamma\left(\frac{k}{2}+1\right)}
  \frac{1+(k+1)\xi}{(1+\xi)^{k+1}} \notag
\\
&= 4 \,\frac{1+2\xi}{(1+\xi)^2}\left[ \frac{1}{2\pi}\, \text{arctan}(\rho)    - \sum_{k=2}^{\infty}  \frac{\rho^{\frac{k(k+1)}{2}}}{\pi^{\frac{k}{2}}}  \frac{1}{\Gamma\left(\frac{k}{2}+1\right)}
  \frac{1+(k+1)\xi}{(1+\xi)^{k-1}(1+2\xi)} \right], \label{eqn:exteljf}
\end{align}
where differentiating under the sum in the first step is allowed due to Lemma~\ref{lem:boundf_k} and we note that
\begin{equation} \label{eqn:f_1explicit}
\kappa_1(\rho) = \frac{1}{2\pi} \int_0^{\infty}  \int_{-\rho z}^0 e^{-x^2/2} e^{-z^2/2} \dd x \dd z=\frac{1}{2\pi} \text{arctan}(\rho)>0.
\end{equation}
To see the last equality, differentiate with respect to $\rho$. Now note that, for $k\geq 2$, the function
$$
\xi \mapsto  \frac{1+(k+1)\xi}{(1+\xi)^{k-1}(1+2\xi)}
$$
is strictly decreasing, which can be see by differentiation. 
Therefore, we can replace it in (\ref{eqn:exteljf}) by its value at $\xi=0$ and thus get the estimate:
\begin{align*}
\Psi_\rho(\xi) - \xi \Psi_\rho^\prime(\xi)
& \geq  4 \,\frac{1+2\xi}{(1+\xi)^2}\left[ \frac{1}{2\pi}\, \text{arctan}(\rho)    - \sum_{k=2}^{\infty}  \frac{\rho^{\frac{k(k+1)}{2}}}{\pi^{\frac{k}{2}}}  \frac{1}{\Gamma\left(\frac{k}{2}+1\right)} \right]
\\
& \geq  4 \,\frac{1+2\xi}{(1+\xi)^2}\left[ \frac{1}{2\pi}\, \text{arctan}(\rho)    - \sum_{k=2}^{\infty}  \frac{\rho^{k+1}}{\pi^{\frac{k}{2}}}   \right]
\\
&=  4 \,\frac{1+2\xi}{(1+\xi)^2}\left[ \frac{1}{2\pi}\, \text{arctan}(\rho)    -\frac{\rho^3}{\sqrt{\pi}(\sqrt{\pi}-\rho)}\right],
\end{align*}
which is strictly positive for all $\rho\in(0,1]$ and all $\xi\geq 0$.
\end{proof}

\begin{corollary}
    The largest solution of (\ref{eqn:eigenvalueeqn2}) admits the following representation
    \begin{equation} \label{eqn:eqn9buermann}
        \lambda = \frac{1}{2} +\frac{1}{2} \sum_{n=1}^\infty \frac{1}{n!} \left. \left( \frac{\partial^{n-1}}{\partial \xi^{n-1}}\left[ \Psi_\rho(\xi)^n\right] \right)\right|_{\xi=0}.
    \end{equation}
\end{corollary}

\begin{proof}
    We use the Lagrange-B\"urmann formula, a variant of the Lagrange inversion formula. Knowing that $\Gamma_\rho(0)=0$, $\Gamma_\rho(\infty)=\infty$, and that it is strictly increasing, we can resolve the equation $\Gamma_\rho(\xi)=1$, i.e.\ compute the inverse of $\Gamma_\rho$ at $1$. The result from the Lagrange-B\"urmann formula is precisely the statement of the corollary. Note that (\ref{eqn:eigenvalueeqn2}) admits several solutions in $\lambda$, while here we show that there is only one solution in $\xi\geq 0$, i.e.\ one solution in $\lambda$ with $\lambda\geq \frac{1}{2}$.
\end{proof}

\begin{corollary}
    The largest solution of (\ref{eqn:eigenvalueeqn2}) admits the following representation
    \begin{equation} \label{eqn:newrepresentationA}
\lambda=\frac{1}{2}+  \sum_{n=1}^{\infty}\frac{(-1)^{n-1}2^{n-1}}{n!}\sum_{k_1=1}^{\infty}\ldots  \sum_{k_{n}=1}^{\infty} \left(\prod_{j=1}^{n}  2^{k_j} \kappa_{k_j}(\rho) \right)  \frac{ \Gamma(k_1+\ldots+k_{n}+n-1 )}{\Gamma(k_1+\ldots+k_{n}) }.
    \end{equation}
\end{corollary}

\begin{proof}
To compute the successive derivatives in (\ref{eqn:eqn9buermann}), we write $\Psi_\rho$ as a Laplace transform
$$
\Psi_\rho(\xi) =\sum_{k=1}^{\infty}  2^{k+1} \kappa_k(\rho)  \frac{1}{\Gamma(k)}\int_0^{\infty}e^{-\xi x} e^{-x} x^{k-1}    \dd x= \int_0^{\infty}  e^{-\xi x } \psi_\rho(x)\dd x,
$$
where we can exchange the sum and integral, due to Lemma~\ref{lem:boundf_k}, and we set
$$
\psi_\rho(x) :=e^{-x}   \sum_{k=1}^{\infty} \frac{2^{k+1}}{\Gamma(k)}  \kappa_k(\rho)  x^{k-1}  .
$$
As a consequence,
\begin{align*}
\left( \frac{\partial^{n-1}}{\partial \xi^{n-1}}  \left(\Psi_\rho(\xi)\right)^n  \right)\bigg|_{\xi=0}
& = \left(\int_0^{\infty} (-1)^{n-1} x^{n-1}e^{-\xi x} \psi_\rho^{\ast(n)}(x) \dd x  \right)\bigg|_{\xi=0}
\\
&=  (-1)^{n-1} \int_0^{\infty}  x^{n-1}\psi_\rho^{\ast(n)}(x) \dd x,
\end{align*}
where $\ast$ denotes the usual convolution product. For the first term, we have
$$
\psi_\rho^{\ast (1)}(x) = \psi_\rho(x)=e^{-x}   \sum_{k_1=1}^{\infty} \frac{2^{k_1+1}}{\Gamma(k_1)}  \kappa_{k_1}(\rho)  x^{k_1-1} .
$$
For the second term, we get
\begin{align*}
\psi_\rho^{\ast (2)}(x)
 &= e^{-x} \int_0^x  \sum_{k_1=1}^{\infty} \frac{2^{k_1+1}}{\Gamma(k_1)}  \kappa_{k_1}(\rho)  (x-y)^{k_1-1}  \sum_{k_2=1}^{\infty} \frac{2^{k_2+1}}{\Gamma(k_2)}  \kappa_{k_2}(\rho)  y^{k_2-1} \dd y
 \\
&= e^{-x}  \sum_{k_1=1}^{\infty} \frac{2^{k_1+1}}{\Gamma(k_1)}  \kappa_{k_1}(\rho)    \sum_{k_2=1}^{\infty} \frac{2^{k_2+1}}{\Gamma(k_2)}  \kappa_{k_2}(\rho)  x^{k_1+k_2-1}\int_0^1 (1-y)^{k_1-1}   y^{k_2-1} \dd y
\\
&= e^{-x}  \sum_{k_1=1}^{\infty} \sum_{k_2=1}^{\infty} 2^{k_1+k_2+2} \kappa_{k_1}(\rho)   \kappa_{k_2}(\rho)  \frac{ x^{k_1+k_2-1}}{\Gamma(k_1+k_2)},
\end{align*}
where exchanging sums and integral is permitted by Lemma~\ref{lem:boundf_k}. By induction, we obtain
$$
\psi^{\ast (n)}(x) = e^{-x} \sum_{k_1=1}^{\infty}\ldots  \sum_{k_{n}=1}^{\infty} 2^{k_1+\ldots+k_{n} +n} \left(\prod_{j=1}^{n}  \kappa_{k_j}(\rho)\right) \cdot   \frac{ x^{k_1+\ldots+k_{n}-1}}{\Gamma(k_1+\ldots+k_{n}) } 
$$
and so
\begin{align*}
 &\int_0^{\infty}  x^{n-1}\psi_\rho^{\ast(n)}(x) \dd x 
 \\
 & = \sum_{k_1=1}^{\infty}\ldots  \sum_{k_{n}=1}^{\infty}  \left(\prod_{j=1}^{n} 2^{k_j+1} \kappa_{k_j}(\rho)\right)  \cdot  \frac{ \Gamma(k_1+\ldots+k_{n}+n -1)}{\Gamma(k_1+\ldots+k_{n}) }. 
\qquad \qedhere{}
 \end{align*}

\end{proof}

Together with the bound in Lemma~\ref{lem:boundf_k}, it is now possible to obtain a bound for the radius of convergence -- despite the fact that we cannot determine the $\kappa_n(\rho)$ more explicitly.

\begin{corollary}
    \label{cor:radius}
    The representation in (\ref{eqn:newrepresentationA}) converges absolutely for $|\rho|<0.332$.
\end{corollary}

\begin{proof}
Observe that the coefficient $\kappa_k(\rho)$ admits a series expansion of the form 
$$
\kappa_k(\rho) =: \frac{\rho^{\frac{k(k+1)}{2}}}{(2\pi)^{\frac{k+1}{2}}} \sum_{n=0}^{\infty} \psi_n^{(k)}\rho^n.
$$
By putting absolute values everywhere, we thus obtain the following bound for the absolute value of the term in (\ref{eqn:newrepresentationA}):
\begin{equation} \label{eqn:beforebinexp}
\frac{1}{2}+\frac{1}{2}  \sum_{n=1}^{\infty}\frac{2^n}{n!} \sum_{k_1=1}^{\infty}\ldots  \sum_{k_n=1}^{\infty} \left(\prod_{j=1}^n   \frac{\rho^{\frac{k_j(k_j+1)}{2}} 2^{k_j}}{(2\pi)^{\frac{k_j+1}{2}}}  \sum_{p=0}^{\infty} |\psi_p^{(k_j)}|\,|\rho|^p\right)  \frac{ \Gamma(k_1+\ldots+k_n+n-1 )}{\Gamma(k_1+\ldots+k_n) }.
\end{equation}
For the remainder of this proof, we consider $\rho>0$ only to avoid the absolute value signs. To compute the multiple sum, we first bound the ratio of Gamma functions. Note that, by the binomial theorem, for $x\in\N$
\begin{align*} 
\frac{\Gamma(x+n-1)}{\Gamma(x)} & = \frac{(x+n-2)!}{(x-1)!} = (n-1)!\, \frac{(x+n-2)!}{(n-1)! (x-1)!} 
\\
& =(n-1)!\, \binom{x+n-2}{n-1} \leq (n-1)!\, 2^{x+n-2}.
\end{align*}

  
As a consequence, also using the Fubini-Tonelli theorem, we get the upper bound for the term in (\ref{eqn:beforebinexp}):
\begin{align}
&\frac{1}{2}+\frac{1}{2} \sum_{n=1}^{\infty} \frac{2^n}{n!}\sum_{k_1=1}^{\infty}\ldots  \sum_{k_n=1}^{\infty}  \left(\prod_{j=1}^n   \frac{\rho^{\frac{k_j(k_j+1)}{2}} 4^{k_j}}{(2\pi)^{\frac{k_j+1}{2}}}  \sum_{p=0}^{\infty} |\psi_p^{(k_j)}|\rho^p\right) (n-1)! \, 2^{n-2}
\notag 
\\
= &  \frac{1}{2}+\frac{1}{8} \sum_{n=1}^{\infty} \frac{4^n}{n} \left( \sum_{k=1}^{\infty}  \frac{\rho^{\frac{k(k+1)}{2}} 4^k}{(2\pi)^{\frac{k+1}{2}}}  \sum_{p=0}^{\infty} |\psi_p^{(k)}|\rho^p\right)^n.
\label{eqn:con1}
\end{align}
We thus arrive at the following sufficient condition for the sum in (\ref{eqn:con1}) to converge:
\begin{equation} \label{eqn:con3}
\sum_{k=1}^{\infty}  \frac{\rho^{\frac{k(k+1)}{2}} 4^k}{(2\pi)^{\frac{k+1}{2}}}  \sum_{p=0}^{\infty} |\psi_p^{(k)}|\rho^p< \frac{1}{4}. 
\end{equation}
We now study $\sum_{p=0}^{\infty} |\psi_p^{(k)}|\rho^p$. For $k=1$, we have due to (\ref{eqn:f_1explicit}):
$$
\kappa_1(\rho) 
= \frac{\text{arctan}(\rho)}{2\pi}
=\frac{1}{2\pi} \sum_{n=0}^{\infty}  (-1)^n \frac{\rho^{2n+1}}{2n+1};
$$
hence, the definition $\displaystyle \kappa_1(\rho) = \frac{\rho}{2\pi} \sum_{p=0}^{\infty} \psi_p^{(1)}\rho^p$ yields
$$
\sum_{n=0}^{\infty} |\psi_n^{(1)}|\rho^n = \sum_{n=0}^{\infty}   \frac{\rho^{2n}}{2n+1}  =\frac{1}{2\rho}\log\left(\frac{1+\rho}{1-\rho}\right).
$$
Turning to the terms for $k\geq 2$ in (\ref{eqn:con3}), similarly to the proof of Lemma~\ref{lem:boundf_k}, we have
\begin{align}
\sum_{n=0}^{\infty} |\psi_n^{(k)}|\rho^n
&\leq  \int_0^{\infty} \int_0^{s_{0}} \ldots \int_0^{s_{k-1}} \int_0^{s_k} \exp\left(\frac{1}{2}\sum_{i=1}^{k} (\rho^{i}s_i)^2\right) e^{-\frac{s_0^2}{2}}   \dd s_k\ldots \dd s_0
\notag \\
&\leq \int_0^{\infty}   \frac{s^{k}}{k!}\exp\left(\frac{1}{2}\sum_{i=1}^{k} (\rho^{i}s)^2\right) e^{-\frac{s^2}{2}}   \dd s
\notag \\
&\leq \int_0^{\infty}   \frac{s^{k}}{k!}\exp\left(-\frac{s^2}{2}\frac{1-2\rho^2}{1-\rho^2}\right)    \dd s
\notag \\
&=\frac{1}{2\cdot k!} \, \Gamma\left(\frac{k+1}{2}\right) 2^{\frac{k+1}{2}}\left(\frac{1-\rho^2}{1-2\rho^2}\right)^{\frac{k+1}{2}}. \label{eqn:estmoreterms4}
\end{align}
Therefore, (\ref{eqn:con3}) is implied by
$$
\frac{1}{\pi}\log\left(\frac{1+\rho}{1-\rho}\right)+    \sum_{k=2}^{\infty}  \frac{\rho^{\frac{k(k+1)}{2}} 4^k}{\pi^{\frac{k+1}{2}}}\frac{1}{2 \cdot k!}  \Gamma\left(\frac{k+1}{2}\right) \left(\frac{1-\rho^2}{1-2\rho^2}\right)^{\frac{k+1}{2}} < \frac{1}{4}. 
$$
One can check numerically that this is true at least for all $\rho<0.332$.
\end{proof}

Note that the bound in the last corollary can be improved by evaluating more terms in the sum (\ref{eqn:con3}) explicitly instead of using the estimate (\ref{eqn:estmoreterms4}).

\begin{lemma}  \label{lem:bothrepresenationsareidentical}
    The representation in (\ref{eqn:newrepresentationA}) holds if and only if the representation $\lambda_\rho=\sum_{i=0}^\infty K_i \rho^i$ holds, where the $(K_i)$ are as in Theorem~\ref{thm: connection to persistence problem MA}.
\end{lemma}

\begin{proof}
First observe that the $\kappa_k(\rho)$ can be written as a series in $\rho$. Now, within the radius of convergence, one may rearrange all sums in (\ref{eqn:newrepresentationA}). Therefore, the representation in (\ref{eqn:newrepresentationA}) may be re-written as a series in $\rho$. Since the representation $\lambda_\rho=\sum_{i=0}^\infty K_i \rho^i$ is analytic in a vicinity of $0$, the coefficients of the two series have to agree.
Therefore, the radius of convergence has to be identical. 
\end{proof}

We mention that a large portion of the results of this manuscript are part of the PhD thesis \cite{Kettnerdissertation}.

%

\end{document}